\documentclass[11pt]{article}
\usepackage{geometry}                
\usepackage{mathtools}
\usepackage{amssymb}
\usepackage{amsmath}
\usepackage{amsthm,dsfont,yfonts}
\usepackage{rotating}
\usepackage{graphicx,pspicture}
\usepackage{array}
\usepackage{calc}
\usepackage{soul}
\usepackage[all]{xy}
\usepackage{amsthm}
\usepackage{lmodern}
\usepackage{color}
\usepackage{pstricks,pst-node,pst-tree}
\usepackage{graphics,graphicx}
\usepackage[british]{babel}
\usepackage{fancyhdr}
\usepackage{color}
\usepackage{multicol}
\usepackage{enumerate}
\usepackage{setspace}

\newtheorem{theorem}{Theorem}
\newtheorem{cor}[theorem]{Corollary}
\newtheorem{lemma}[theorem]{Lemma}
\newtheorem{prop}[theorem]{Proposition}
\newtheorem{remark}[theorem]{Remark}
\theoremstyle{definition}

\definecolor{magenta}{RGB}{203,0,150}
\definecolor{blueish}{RGB}{0,35,211}
\definecolor{greenish}{RGB}{0,100,20}

\newcommand{\h}{\hspace{2mm}}  

\usepackage{makeidx}
\makeindex

\usepackage[noindentafter]{titlesec}
\title{Palindromic words in simple groups}
\date{\today}
\author{Elisabeth Fink\footnote{The first author was supported by the ERC-StG 257110 ``RaWG''} \quad Andreas Thom\footnote{The second author was supported by the ERC-StG 277728 ``GeomAnGroup''}}
  
\begin{document}

\onehalfspace

\selectlanguage{british}

\maketitle

\begin{abstract}
A palindrome is a word that reads the same left-to-right as right-to-left. We show that every simple group has a finite generating
set $X$, such that every element of it can be written as a palindrome in the letters of $X$. Moreover, every simple group
has palindromic width $pw(G,X)=1$, where $X$ only differs by at most one additional generator from any given generating set. On
the contrary, we prove that all non-abelian finite simple groups $G$ also have a generating set $S$ with $pw(G,S)>1$. 

As a by-product of our
work we also obtain that every just-infinite group has finite palindromic width with respect to a finite generating set. This
provides first examples of groups with finite palindromic width but infinite commutator width.

\medskip

\begin{center}

\end{center}

\end{abstract}

\section{Introduction}

Words in simple groups have been intensively studied during the last decades. A well-known example is the proof of the Ore
conjecture \cite{ore_conj}, which states that in every non-abelian finite simple group, every element can be written as a commutator. In
other words, every finite simple group has commutator width $1$. It has later emerged \cite{muranov_simple_commwidth}, that there
exist infinite simple groups with infinite commutator width. A palindromic word is a word which reads the same left-to-right as
right-to-left. The study of the so-called palindromic width in various classes of groups has emerged in the last decade
(see for example \cite{bardakov_nilpotent}, \cite{bardakov}, \cite{bardakov_shpilrain_tolstykh}, \cite{me_palindromic},
\cite{rileySale}). In this paper we show that for both, finite and infinite simple groups, it is possible to add at most one
generator to any given generating set such that every element can be expressed as a single
palindrome with respect to this possibly modified generating set. In contrast to this, we also show that every non-abelian finite simple
group has a generating set under which there exist elements which cannot be expressed as a palindrome. 
We apply our techniques to
just-infinite groups, to see that all of them have a finite generating set with respect to which they have finite palindromic
width.

\section{Palindromes}

Let $F_n$ be the free group on $n$ generators $X:=\left\{x_1, \dots x_n\right\}$. We write $X^{\pm 1} := X \cup X^{-1}$, where
$X^{-1} := \{x_1^{-1},\dots,x_n^{-1}\}$.  An element $g$ of $F_n$ is called a
\emph{palindrome} if there exists some $k \in \mathbb N$ so that $g=\prod_{j=1}^k y_{i}$, where $y_{j}=y_{k-j+1}$ and $y_j \in
X^{\pm 1}$ for all $1 \leq j \leq k$. In other words, if it reads the same left-to-right
as right-to-left. Let $G$ be a group with $n$ generators. Then there exists a natural map $\pi: F_n \to G$. An
element $g$ of $G$ with generating set $S=\left\{s_1, \dots, s_n\right\}$ is called a \emph{palindrome} if at least one of its
preimages in $F_n$ is a palindrome.

\begin{remark}\label{rem_1}
If $G$ is a non-free group then there exists an element $r \in F_n$ with $\pi(r)=1$. Let now $p$ be a palindrome in $G$. By
definition there exists an element $s \in F_n$ that is a palindrome and such that $\pi(s)=p$. On the other hand we also have 
$\pi(sr)=p$, but $sr$ need not be a palindrome anymore. 
\end{remark}

We denote by $\overline{g}$ a \emph{reverse} element of $g \in G$. By this we mean the following: Let $S=\left\{s_1,
\dots, s_n\right\}$ be a generating set for $G$ and $g=\prod_{i=1}^k t_{i}$ be a presentation for $g$ with $t_i \in S^{\pm 1}$.
Then
$\overline{g}$ is given by $\overline{g} = \prod_{i=1}^k t_{k-i+1}$. 

\medskip

Every element of a group can always be expressed as a product of palindromes. Indeed, as every generator and every inverse of a generator is a palindrome by definition every element $g$ of $G$ can be written as a product of finitely many
palindromes, where the number of palindromes depends in general on $g$. We are interested in cases when this number can be bounded
independently of $g \in G$ or is equal to one for some generating sets.

For a group $G$ denote by $P_{G,X}$ the set of all palindromes with respect to the generating set $X$. We define the
\emph{palindromic
width} of $G$ with respect to a generating set $X$ as \[pw(G,X)=\sup_{g \in G}\left\{ \min_{k} \left\{ k \h | \h
g = \prod_{i=1}^k p_i, \h p_i \in P_{G,X}\right\}\right\}.\] 

This number might be infinite, however there are certain classes of groups in which this is finite. These are for example certain
extensions of nilpotent and solvable groups (\cite{bardakov_soluble}, \cite{bardakov}) or certain wreath products
(\cite{me_palindromic}, \cite{rileySale}).

\section{Relations and Words}

In this section we show that careful analysis of given relations in a group gives rise to a natural set of palindromes. We  use
this to prove our main result. Further, we use the Feit-Thompson theorem to show that non-abelian finite simple groups have a
generating set under which not every element can be written as a palindrome. This has already been used by the first author in
\cite{me_palindromic}.

\begin{lemma}[see \cite{me_palindromic}]\label{lemma_change_gen_set}
Assume $G$ is a finitely generated non-abelian group with respect to the generating set $X$, that has a relation in which
two non-commuting generators occur. Then $G$ has a
presentation with generating set $\tilde{X}$, in which there exists a relation $q$ in $G$ such
that $\overline{q}\neq 1$ in
$G$, where $\tilde{X}$ is obtained from $X$ by adding at most one generator.
\end{lemma}

\begin{proof}
Assume that for the pair of non-commuting generators $x$ and
$y$, there exists a relation $r$, in which $x$ and $y$ occur. Without loss of generality we can assume that $x$ and $y$ occur as
the subword $xy$ at the beginning of $r$. So we have
\[r=xy w=1,\]  where $w$ is another word in the generators of $G$. If also $\overline{r} = \overline{w} yx =1$, then we add a new
generator $c=xy$ and get
\[r = c w, \quad \overline{r} = \overline{w} c.\]
If we assume that also $\overline{r}=\overline{w}c=1$, then we have the following
\[\overline{w} c = \overline{w}xy = \overline{w}yx,\] which implies $[x,y]=1$, contradicting our assumption that $x$ and $y$ were
two non-commuting generators.
\end{proof}

\begin{prop}\label{prop_normal}
For any group $G$, the set of all palindromes $N:=\left\{\overline{w}w
\h | \h w=1\right\}$ is a normal subgroup of $G$.
\end{prop}

\begin{proof}
Let $\overline{t}t, \overline{s}s$ be two elements of $N$. Their product can be written as
\[\overline{t}t\cdot \overline{s}s = \overline{t}\overline{s}\cdot st,\] because we have that
$s=t=1$. It is again of the form
$\overline{(st)}(st)$ with $st=1$, hence an element of $N$. A similar argument shows that $N$ is closed under inverses.
If $g \in G$, then \[g^{-1}
\overline{t}tg = g^{-1}\overline{t}tg
\overline{g}\overline{g^{-1}}= g^{-1}\overline{t}g\overline{g}t
\overline{g^{-1}} = \overline{(\overline{g} t\overline{g^{-1}})} \cdot (\overline{g}t\overline{g^{-1}}),\] which is again an
element
of $N$ because
$\overline{g}t\overline{g^{-1}}=\overline{g}\cdot 1
\cdot \overline{g^{-1}}=1$.
\end{proof}

\begin{theorem}
Let $G$ be a non-abelian simple group generated by $X$. Then we have that $pw(G,\tilde{X})=1$, where $\tilde{X}$ is obtained from
$X$ by adding at most one generator.
\end{theorem}

\begin{proof}
Lemma \ref{lemma_change_gen_set} states that there exists a generating set $\tilde{X}$ which is obtained from $G$ by adding at
most one generator, such that we have a relation $q$ with $\overline{q} \neq 1$. This implies that the normal subgroup $N$
from Proposition \ref{prop_normal} is not $\{1\}$. Because $G$ is simple, we have $N=G$ and every element of $G$ can be
written as a single palindrome with respect to $\tilde{X}$.
\end{proof}

In \cite{muranov_simple_commwidth} A. Muranov constructs a finitely generated infinite simple group of unbounded commutator width. This leads to
first examples of groups which have finite palindromic width $1$, but do not have finite commutator width. 

\begin{cor}
There exists a finitely generated simple group with infinite commutator width but finite palindromic width with respect to some finite generating set.
\end{cor}

A group $G$ is called \emph{just-infinite} if for every non-trivial normal subgroup $N \lhd G$, the quotient $G/N$ is finite.

\begin{cor}\label{thm_JI}
Let $G$ be a just-infinite group with generating set $X$. Then $pw(G,\tilde{X})<\infty$, where $\tilde{X}$
which is obtained from $X$ by adding at most one generator.
\end{cor}

\begin{proof}
By Lemma \ref{lemma_change_gen_set} we can ensure by passing from $X$ to $\tilde{X}$ that there exists a non-trivial element $g
\in G$ such
that $\overline{g}$=1. Hence $N = \left\{\overline{g}g \h | \h \overline{g}\neq 1, g=1\right\}$ is a normal subgroup
of $G$ different from $\{1\}$. Because $G$ is just-infinite, $N$ has finite index. Take $r$ to be a coset-representative of
shortest length $l_r$. Every
element $g$ of $G$ can be written as $g=r \cdot n$, where $n \in N$, with $n$ a palindrome. Hence every element of $G$ is a
product of at most $l_r+1$ palindromes. 
\end{proof}

The following theorem demonstrates that such a change of the generating set is necessary in certain cases.

\begin{theorem}
Let $G$ be a non-abelian finite simple group. Then $G$ has a generating set $X$ such that
$pw(G,X)>1$. 
Moreover, there exist generating sets $X_n$ of the alternating groups $A_n$, such that $pw(A_n,X_n)\geq n/4$ for all $n \geq 5$.
\end{theorem}

\begin{proof}
By the theorem of Feit-Thompson we have that $G$ must contain at least one element $g$ of order $2$. Then the normal closure
$\left<\left<g\right>\right>$ of $g$ in $G$ is equal to $G$. Hence $G$ has a generating set consisting of elements of order $2$.
Let $X$ be this generating set with generators $x_1, \dots, x_n$. Now if $p$ is a palindrome of even length, then $p$ has the form
$\overline{k}x_ix_ik$, for some $x_i \in X$ and $k \in G$. But each $x_i$ has order $2$, hence all palindromes of even length are
trivial. Because every generator has order two, we have the special case that $\overline{g}=g^{-1}$, so every palindrome is a
conjugate of a single generator, and as such must have order $2$ as well. If $G$ is non-abelian simple, then $G$
contains an element of odd order, because no non-abelian group $H$ with $|H|=2^n, n>1$, can be simple. However, if now $h$ is an element of
odd order in $G$,
it cannot be a conjugate of a generator $x_i$ and hence cannot be written as a single palindrome in the generators $x_1, \dots,
x_n$. This proves the first claim.

Consider now the alternating group $A_n$ and the permutation
$$\sigma_n:=(12)(34)(5) \cdots (n) \in A_n.$$ We set $X_n$ equal to the conjugacy class of $\sigma_n$. As noted above, $X_n$ generates $A_n$ and each palindrome is conjugate to $\sigma_n$. Thus, it is clear that any product of $k$ palindroms must fix at least $n-4k$ points and hence $pw(A_n,X_n) \geq n/4$. This proves the second claim.
\end{proof}

In particular, we can conclude from the previous theorem that there is no uniform upper bound on the palindromic width over all
non-abelian finite simple groups and all generating sets. However, the situation is different if the rank (in the sense of algebraic groups, i.e., the dimension of the Cartan subgroup of the associated algebraic group over an algebraically closed field -- or in case of $A_n$ just $n-1$) of the non-abelian
finite simple group {\it or} the size of the generating set remains bounded.

\begin{theorem}
There exists a constant $C>0$ such that for every non-abelian finite simple group $G$ of rank $r$ (in the sense of algebraic groups) and any generating set $X
\subset G$ we have
$$pw(G,X) \leq C \cdot \min\{ r, |X| \}.$$
\end{theorem}
\begin{proof}
Let $g \in G$ and $w \in P_{G,X}$. Then, 
$gwg^{-1} = g \overline g \cdot \overline{g^{-1}}wg^{-1},$
and hence $gwg^{-1} \in P_{G,X}^2$. Thus,
$$\bigcup_{g \in G} gP_{G,X}g^{-1} \subset P_{G,X}^2.$$
Thus, we can bound the palindromic width by twice the number of steps one needs to generate $G$ by the conjugacy class of some
element in $X$. Using the main result in \cite{liebeck_shalev}, see also \cite{stolz_thom}, this implies the claim.
\end{proof}

\renewcommand{\bibname}{References}
\bibliographystyle{plain}  
\bibliography{bibliography}        

\medskip
\medskip

\begin{multicols}{2}
\textbf{Elisabeth Fink}\\
DMA\\
\'Ecole Normale Sup\'erieure\\
45 rue d'Ulm\\ 
75005 Paris, France\\
+33 (0)1 44 32 21 78\\
fink@ens.fr

\medskip
\medskip

\textbf{Andreas Thom}\\
Mathematisches Institut\\
Universit\"at Leipzig\\
PF 10 09 20\\
D-04009 Leipzig, Germany\\
+49 (0)341 9732185\\
andreas.thom@math.uni-leibzig.de
\end{multicols}

\end{document}